\documentclass{article}
\usepackage{arxiv}
\usepackage[utf8]{inputenc} 
\usepackage[T2A]{fontenc}    
\usepackage[english,russian]{babel}
\usepackage{mmap}
\usepackage{hyperref}       
\usepackage{url}            
\usepackage{booktabs}       
\usepackage{amsfonts}       
\usepackage{nicefrac}       
\usepackage{amsmath,amssymb,amsthm}
\usepackage{mathtools}
\usepackage{algorithm, algorithmic}
\floatname{algorithm}{Алгоритм}
\usepackage{graphicx}
\usepackage{stackengine}    
\usepackage[labelsep=period]{caption}
\usepackage{hyperref}
\hypersetup{
  colorlinks   = true,    
  urlcolor     = green,   
  linkcolor    = blue,    
  citecolor    = red      
}
\usepackage{comment}
\usepackage{multirow}
\makeatletter
\usepackage{svg}
\usepackage{titlesec}
\titlelabel{\thetitle. }

\def\keywordname{{\bfseries Ключевые слова:}}%
\def\keywords#1{\par\addvspace\medskipamount{\rightskip=0pt plus1cm
\def\and{\ifhmode\unskip\nobreak\fi\ $\cdot$ }\noindent\keywordname\enspace\ignorespaces#1\par}}
\DeclareMathOperator*{\argmin}{arg\,min}

\theoremstyle{plain}
\newtheorem{theorem}{Теорема}
\newtheorem{lemma}{Лемма}
\newtheorem{proposition}{Предложение}

\newtheorem{definition}{Определение}

\theoremstyle{definition}
\newtheorem{remark}{Замечание}

\renewcommand{\phi}{\varphi}
\renewcommand{\epsilon}{\varepsilon}
\renewcommand{\leq}{\leqslant}
\renewcommand{\geq}{\geqslant}

\author{
{\bfseries Г.\,В.\,Айвазян\,$^{\it{a},\,*}$,
Ф.\,С.\,Стонякин\,$^{\it{a},\,\it{b},\,**},$} \\
{\bfseries Д.\,А.\,Пасечнюк\,$^{\it{a},\,\it{c},\,\it{***}}$,
М.\,С.\,Алкуса\,$^{\it{a},\,\it{d},\,****}$
А.\,М.\,Райгородский\,$^{\it{a},\,\it{e}, \it{f} \,*****}$}
\\ {\itshape $^{\it{a}}$\,Московский физико-технический институт,}
\\ {\slshape 141701,} {\itshape г. Долгопрудный, Институтский пер.,~9}
\\ {\itshape $^{\it{b}}$\,Крымский федеральный университет им. В.\,И.\,Вернадского,}
\\ {\slshape 295007,} {\itshape г.~Симферополь, проспект академика Вернадского,~4}
\\ {\itshape $^{\it{c}}$\,Исследовательский центр доверенного искусственного интеллекта ИСП РАН,}
\\ {\slshape 109004,} {\itshape г.~Москва, ул. Александра Солженицына,~25}
\\ {\itshape $^{\it{d}}$\,Национальный исследовательский университет <<Высшая школа экономики>>,}
\\ {\slshape 101000,} {\itshape г. Москва, ул. Мясницкая, д.~20}
\\ {\itshape $^{\it{e}}$\,Московский государственный университет им. М. В. Ломоносова,}
\\ {\itshape механико-математический факультет,}
\\ {\slshape 119991,} {\itshape г. Москва, Ленинские горы, д. 1}
\\ {\itshape $^{\it{f}}$\,Кавказский математический центр Адыгейского государственного университета,}
\\ {\slshape 385000,} {\itshape г. Майкоп, ул. Первомайская, 208}
\\ {\itshape *\,e-mail: aivazian.grigory25@yandex.ru, ORCID: 0009-0008-2625-9332}
\\ {\itshape **\,e-mail: fedyor@mail.ru, ORCID: 0000-0002-9250-4438}
\\ {\itshape ***\,e-mail: dmivilensky1@gmail.com, ORCID: 0000-0002-1208-1659}
\\ {\itshape ****\,e-mail: mohammad.alkousa@phystech.edu, ORCID: 0000-0001-5470-0182}
\\ {\itshape *****\,e-mail: raigorodsky@yandex-team.ru, ORCID: 0000-0001-8614-9612}
}

\title{АДАПТИВНЫЙ ВАРИАНТ АЛГОРИТМА ФРАНК--ВУЛЬФА ДЛЯ ЗАДАЧ ВЫПУКЛОЙ ОПТИМИЗАЦИИ}

\renewcommand{\shorttitle}{АДАПТИВНЫЙ ВАРИАНТ АЛГОРИТМА ФРАНК--ВУЛЬФА ДЛЯ ЗАДАЧ ВЫПУКЛОЙ ОПТИМИЗАЦИИ}

\date{}

\begin{document}

\maketitle

\begin{abstract}
В данной работе исследовался вариант метода Франк--Вульфа для задач выпуклой оптимизации с адаптивным подбором параметра шага, соответствующего информации о гладкости целевой функции (константа Липшица градиента). Получены теоретические оценки качества выдаваемого методом приближённого решения с использованием адаптивно подобранных параметров $L_k$. На классе задач на выпуклом допустимом множестве с выпуклой целевой функцией гарантируемая скорость сходимости предложенного метода сублинейна. Рассмотрено сужение этого класса задач (целевая функция с условием градиентного доминирования и получена оценка скорости сходимости с использованием адаптивно подбираемых параметров $L_k$). Важной особенностью полученного результата является проработка ситуации, при которой можно гарантировать после завершения итерации уменьшение невязки функции не менее чем в 2 раза. В то же время использование адаптивно подобранных параметров в теоретических оценках позволяет применять метод как для гладких, так и для негладких задач при условии выполнения критерия выхода из итерации. Для гладких задач можно доказать, что теоретические оценки метода гарантированно оптимальны с точностью до умножения на постоянный множитель. Выполнены вычислительные эксперименты, и проведено сравнение с двумя другими алгоритмами, в ходе чего была продемонстрирована эффективность алгоритма для ряда как гладких, так и негладких задач.
\end{abstract}

\section*{Введение}\label{sec1_introduction}

Метод типа условного градиента впервые был рассмотрен для квадратичных оптимизационных задач на многогранниках в \cite{Canon1968Frank}. Принципиальное отличие метода условного градиента (Франк--Вульфа) от классических методов градиентного спуска состоит в том, что вспомогательная задача является линейной. Это удобно для допустимых множеств, для которых эффективно решать вспомогательные линейные подзадачи минимизации. В последние годы методы условного градиента получили широкое распространение в различных задачах анализа данных со структурными ограничениями. В этом контексте можно отметить задачу регрессии LASSO, бинарную классификацию методом опорных векторов (SVM) и др. \cite{Bomze2021Frank}.

Обзор известных к настоящему времени результатов по методам типа условного градиента имеется, например, в недавно вышедшей книге \cite{Braun2022Conditional}. В частности, известно, что для выпуклых гладких задач оптимальной оценкой скорости сходимости метода Франк--Вульфа является $O\left(\frac{1}{k}\right)$, а для негладких задач метод может расходиться (см. \cite[Example 1]{Nesterov_Non-smooth}). В данной работе предложен адаптивный вариант метода Франк--Вульфа по аналогии с работами Ю.\,Е.\,Нестерова \cite{Nesterov2015Universal}. Этот метод, по-прежнему, оптимален для класса выпуклых гладких задач. А за счёт адаптивной настройки потенциально может быть применен и к негладким задачам. Мы отправляемся от \cite{backtracking2020line-search}, где для гладких выпуклых задач предложен метод первого порядка с адаптивным подбором параметров, соответствующих константе гладкости $L$ целевой функции. Однако в нашей работе впервые детально описаны результаты о теоретической оценке качества выдаваемого решения с использованием адаптивно подбираемых параметров, включая ситуацию, когда невязка функции уменьшается не менее чем в 2 раза на одной итерации. Более того, полученные результаты о качестве выдаваемого решения с адаптивно подбираемыми параметрами потенциально можно применить и к негладким выпуклым задачам, что численно проиллюстрировано в настоящей работе некоторыми примерами.

Работа состоит из введения и трёх разделов. В первом разделе приводится постановка задачи, а также описывается классический метод Франк--Вульфа. Во втором разделе описывается исследуемый адаптивный вариант алгоритма Франк--Вульфа, а также проводится теоретический анализ его сходимости для некоторых типов допущений о целевой функции и допустимом множестве задачи. В третьем разделе продемонстрированы результаты проведенных вычислительных экспериментов и сделаны выводы о практической эффективности предложенного алгоритма. В частности, рассмотрены задачи Ферма--Торричелли--Штейнера и о наименьшем покрывающем шаре, которые относятся к разряду негладких, а также приведены результаты для ряда гладких задач: минимизация взвешенной суммы квадратов компонент вектора, восстановление матрицы, задача классификаторов методом опорных векторов, а также задача логистической регрессии.

\section{Общая схема алгоритма Франк--Вульфа}

Будем рассматривать следующую задачу:
\begin{equation}\label{eq1}
    \min_{x \in Q} f(x),
\end{equation}
где $Q$~--- выпуклое и компактное (т.~е. ограниченное и замкнутое) подмножество $\mathbb{R}^n$ и $f$~--- выпуклая (суб)дифференцируемая функция и $\nabla f(x)$ --- её (суб)градиент в т. $х \in Q$. Всюду далее $x^*$ обозначает решение задачи \eqref{eq1}, $f^* := f(x^*)$~--- соответствующее оптимальное значение $f$.

\subsection{Классический метод Франк--Вульфа}

Классический метод Франк--Вульфа для минимизации целевой функции $f$ порождает
последовательность допустимых точек ${x_k}$, соответствующую алгоритму \ref{classical_FW}.

\begin{algorithm}[!ht]
\caption{Классический метод Франк--Вульфа (метод условного градиента).}\label{classical_FW}
\begin{algorithmic}[1]
   \REQUIRE Количество итераций $N$, начальная точка $x_{0} \in Q$.
   \FOR{$k=0, 1, \ldots, N-1$}
    \STATE Выбрать некоторое $\alpha_k \in [0, 1]$,
    \STATE $s_k = \arg\min_{x \in Q}\left\{  \nabla f(x_k)^\top x \right\}$,
    \STATE $x_{k+1} = (1-\alpha_k) x_k + \alpha_k s_k$,
    \ENDFOR
\ENSURE $x_N$.
\end{algorithmic}
\end{algorithm}
На $k$-й итерации метод выдаёт точку, минимизирующую cкалярное произведение $g^\top z$ ($g$ --- вектор) на допустимом множестве $Q$ вида
\begin{equation} \label{eq2}
    \operatorname{LMO}_Q (g) \in \argmin_{z \in Q} g^\top z,
\end{equation}
и направление спуска определяется как
\begin{equation} \label{eq3}
     d_k = d_k^{FW}: = s_k-x_k, s_k \in \operatorname{LMO}_Q (\nabla f (x_k)).
\end{equation}
В частности, обновление в п.~4 листинга алгоритма \ref{classical_FW} может быть записано в виде
\begin{equation*} 
     x_{k + 1} = x_k + \alpha_k (s_k-x_k) = x_k + \alpha_k d_k.
\end{equation*}

В дальнейшем нам понадобится следующее определение.

\begin{definition}
Дифференцируемая функция $f: Q \longrightarrow \mathbb{R}$ называется $L$-гладкой ($L > 0$) относительно некоторой нормы $\|\cdot\|$, если для любых $x, y \in Q$
\begin{equation}\label{smoothness}
    f(y) - f(x) \leq \langle \nabla f(x), y-x \rangle + \frac{L}{2} \|y-x\|^2.
\end{equation}
\end{definition}
Для выпуклых функций условие \eqref{smoothness} эквивалентно тому, что $f$ имеет $L$-липшицев градиент \cite{Braun2022Conditional}, т.е.
для двойственной к $\|\cdot\|$ нормы $\|\cdot\|_*$ верно $$ \|\nabla f(x) - \nabla f(y) \|_*\ \leq L\|x-y\| \;\; \forall x, y \in Q.$$

\subsection{Известные стратегии выбора параметра шага}
Напомним известные правила выбора параметра шага $\alpha_k \in [0, 1]$ для алгоритма Франк--Вульфа \ref{classical_FW} (см., например, \cite{Bomze2021Frank}).
\begin{enumerate}
     \item Убывающий шаг
     \begin{equation} \label{eq5}
     \alpha_k = \frac{2}{k+2} \quad \forall k \geq 0.
     \end{equation}
     Это правило часто используется.

     \item Точный линейный поиск
     \begin{equation} \label{eq6}
        \begin{aligned}
            \alpha_k = \argmin_{\alpha \in [0, \alpha_k^{max}]} \phi (\alpha),
            \\
            \text{ где }\;\; \phi (\alpha): = f (x_k + \alpha d_k ),
        \end{aligned}
     \end{equation}
     где шаг $ \alpha $ выбирается минимизацией $ \phi $.

     \item Шаг Армихо: при этом подходе итеративно сокращается размер шага, чтобы гарантировать достаточное снижение целевой функции. Он представляет собой хороший способ заменить точный линейный поиск в тех случаях, когда он становится слишком дорогостоящим. На практике устанавливаются параметры $ \delta \in (0,1) $ и $ \gamma \in \left (0, 1/2 \right) $, затем пробуются шаги $\alpha = \delta^m \alpha_k^{max}$,
     \begin{equation} \label{eq7}
     f (x_k + \alpha d_k) \leq f (x_k) + \gamma \alpha \nabla f (x_k)^\top d_k,
     \end{equation}
     и устанавливается $ \alpha_k = \alpha $.

     \item Шаг, который определяется константой Липшица градиента целевой функции.
     \begin{equation} \label{eq8}
     \alpha_k = \alpha_k (L): = \min \left \{-\frac{\nabla f (x_k)^\top d_k}{L \| d_k \|^2}, \alpha_k^{max} \right \},
     \end{equation}
     где $L$~--- константа Липшица $\nabla f$. Размер соответствующего шага может рассматриваться как результат минимизации квадратичной модели $ m_k (\cdot\,; L)$, переоценивающей $f$ вдоль прямой $x_k + \alpha d_k $,
     \begin{equation*} \label{eq9}
        \begin{aligned}
            m_k (\alpha; L) & = f (x_k) + \alpha \nabla f (x_k)^\top d_k + \frac{L \alpha^2}{2} \| d_k \|^2
            \\& \geq f (x_k + \alpha d_k),
        \end{aligned}
     \end{equation*}
     где последнее неравенство следует из стандартной леммы о спуске. Интерес такого выбора шага для метода Франк--Вульфа заключается, в частности в возможности предложить достаточные условия убывания на итерации невязки по функции не менее чем в 2 раза.
\end{enumerate}

\section{Теоретический анализ адаптивного варианта алгоритма Франк--Вульфа}

\subsection{Слyчай выпyклых задач}

Отталкиваясь от правила выбора шага \eqref{eq8}, рассмотрим получаем адаптивный вариант алгоритма Франк--Вульфа (см. алгоритм \ref{adaptive_FW} ниже). Его отличительная особенность --- подбор <<локальных>> значений параметров $L_k$ посредством проверки некоторых неравенств на итерациях метода.

\begin{algorithm}[!ht]
\caption{Адаптивный метод Франк--Вульфа.}\label{adaptive_FW}
\begin{algorithmic}[1]
   \REQUIRE Количество итераций $N$, начальная точка $x_{0} \in Q, L_{-1}>0$.
   \FOR{$k=0, 1, \ldots, N-1$}
    \STATE $L_{k}=\frac{L_{k-1}}{2}$.
    \STATE $s_k = \argmin_{x \in Q}\left\{ \nabla f(x_k)^\top x \right\}$.
    \STATE $d_k = s_k - x_k$.
    \STATE $\theta_k = \min \left \{- \frac{ \nabla f(x_k)^\top d_k }{L_k \|d_k\|^2}, 1 \right \} $.
    \IF{$\theta_k < 1$}
    \IF{$f(x_k + \theta_k d_k) \leq f(x_k) - \frac{\left( \nabla f(x_k)^\top d_k \right)^2}{2L_k \|d_k\|^2}$. }
    \STATE $\alpha_k := \theta_k$,
    \STATE $x_{k+1} = x_k + \alpha_k d_k$.
    \ELSE
    \STATE $L_{k}=2 L_{k}\text{ и переходим к п.~4}.$
    \ENDIF
    \ELSE
    \IF{$f(x_k + d_k) \leq f(x_k) +  \nabla f(x_k)^\top d_k + \frac{L_k}{2}\|d_k\|^2$, }
    \STATE $\alpha_k := 1$,
    \STATE $x_{k+1} = x_k + \alpha_k d_k$.
    \ELSE
    \STATE $L_{k}=2 L_{k}\text{ и переходим к п.~4}.$
    \ENDIF
    \ENDIF
 \ENDFOR
\ENSURE $x_N$.
\end{algorithmic}
\end{algorithm}

\begin{remark}\label{remark_1}
Если $f$ имеет $L$-липшицев градиент и $L_{-1} \leq 2L$, то из \cite[Lemma 1]{Bomze2021Frank} ясно, что $L_k \leq 2L $.
\end{remark}
\begin{remark}\label{remark_2}
Пусть $f$ имеет $L$-липшицев градиент и на шаге $k$ было осуществлено $i_k$ проверок неравенства из п.~7 или п.~14. Тогда при условии $L_{-1} \leq 2L$ для алгоритма \ref{adaptive_FW} в силу замечания \ref{remark_1} имеет место следующий факт:
\begin{align*}
        &i_0 + i_1 +\ldots + i_N =\\
        &\quad= \left(2 + \log_{2}{\frac{L_1}{L_0}}\right)+ \ldots +\left(2 +\log_{2}{\frac{L_N}{L_{N-1}}}\right)
        \\&\quad= 2N + \log_{2}{\left(\frac{L_1}{L_0}\frac{L_2}{L_1} \cdots \frac{L_N}{L_{N-1}}\right)}
        \\&\quad= 2N + \log_{2}{\frac{L_N}{L_0}}   \leq 2N + \log_2{\frac{2L}{L_0}} = O(N).
\end{align*}
Таким образом, общее количество проверок неравенства из п.~7 или п.~14 после $N$ итераций алгоритма \ref{adaptive_FW} составляет $O(N)$.
\end{remark}

Приведем теоретические результаты о сходимости алгоритма \ref{adaptive_FW}.

Рассмотрим следующий параметр (см., например, \cite{Bomze2021Frank}), который часто рассматривают в качестве меры сходимости, известный как зазор двойственности:
\begin{equation}\label{eq10}
    G(x) = \max_{s \in Q} \nabla f(x)^\top(x-s).
\end{equation}

Этот параметр всегда неотрицателен и равен нулю только в стационарных точках $f$. Если $f$~--- выпуклая функция, то, используя, что $\nabla f(x)$~--- это субградиент, мы достигаем следующего:
\begin{equation}\label{eq11}
    G(x) \geq -\nabla f(x)^\top(x-x^*) \geq f(x) - f^*.
\end{equation}
Поэтому $G(x)$~--- это верхняя граница зазора оптимальности в точке $x$. Отметим, что $G(x)$~--- это частный случай зазора двойственности Фенхеля.

Докажем вспомогательную лемму, приводящую к достаточным условиям убывания на итерации алгоритма невязки по функции не менее чем в 2 раза.
\begin{lemma}\label{lemma_1}
Если $f$~--- выпуклая функция и в алгоритме \ref{adaptive_FW} $\alpha_k = 1$ и $d_k = d_k^{FW}$, то при любых $k \geq 0$
\begin{equation}\label{eq12}
    f(x_{k+1}) - f^* \leq \frac{1}{2} \min \left(L_k \|d_k\|^2, f(x_k) - f^*\right).
\end{equation}
\end{lemma}
\begin{proof}
Если $\alpha_k = 1$, тогда по определениям \eqref{eq3} и \eqref{eq10} верно, что
\begin{equation}\label{eq13}
    G(x_k) = -\nabla f(x_k)^\top d_k \geq L_k \|d_k\|^2,
\end{equation}
где последнее неравенство верно в связи с предположением, что $\alpha_k = 1$, а значит, $\theta_k \geq 1$. Из листинга алгоритма \ref{adaptive_FW} следует, что при $\alpha_k = 1$ верно следующее:
\begin{equation}\label{eq14}
    \begin{aligned}
        & f(x_{k+1}) - f^* = f(x_k+d_k) - f^* \leq
        \\& \;\; \leq f(x_k) - f^* + \nabla f(x_k)^\top d_k + \frac{L_k}{2} \|d_k\|^2.
    \end{aligned}
\end{equation}
Учитывая определение $d_k$ и выпуклость $f$, получаем
\begin{equation*}
    \begin{aligned}
    &f(x_k) - f^* + \nabla f(x_k)^\top d_k \leq
    \\&\;\; \leq f(x_k) - f^* + \nabla f(x_k)^\top(x^* - x_k) \leq 0.
    \end{aligned}
\end{equation*}
Отсюда, ввиду \eqref{eq14}, верно, что
$$
    f(x_{k+1}) - f^* \leq \frac{L_k}{2} \|d_k\|^2.
$$
На базе неравенства \eqref{eq14} имеем
\begin{equation}\label{eq15}
    \begin{aligned}
        & f(x_k) - f^* + \nabla f(x_k)^\top d_k + \frac{L_k}{2} ||d_k||^2 \leq
        \\& \leq f(x_k) - f^* - \frac{1}{2}G(x_k) \leq \frac{f(x_k)-f^*}{2},
    \end{aligned}
\end{equation}
где в первом неравенстве было использовано \eqref{eq13}, а во втором  $G(x_k) \geq f(x_k) - f^*$.
\end{proof}

Сформулируем теперь основной результат данного раздела.

\begin{theorem}\label{MainThm}
Пусть $f$~--- выпуклая функция. Тогда для алгоритма \ref{adaptive_FW}, при любых $k \geq 1$ верно
\begin{equation}\label{eq16}
    f(x_k) - f^* \leq \frac{2 D^2 \max_{j \in \overline{0,k-1}} L_j}{k+2}.
\end{equation}
Если $\alpha_k = 1$, то
\begin{equation}\label{eq17}
    f(x_k) - f^* \leq \frac{ f(x_{k - 1}) - f^*}{2}.
\end{equation}
\end{theorem}
\begin{proof}
 Докажем \eqref{eq16} индукцией по $k$. Для $k = 0$ и $\alpha_0 = 1$ по лемме \ref{lemma_1} верно, что
$$
    f(x_1) - f^* \leq \frac{L_0 \|d_0\|^2}{2} \leq \frac{L_0 D^2}{2}.
$$
Если $\alpha_0 < 1$, то
$$
f(x_0) - f^* \leq G(x_0) \leq L_0 \|d_0\|^2 \leq L_0 D^2.
$$
Таким образом, базис индукции верен. Из \eqref{eq12} очевидно, что для всех $k$, таких, что $\alpha_k = 1$, шаг индукции выполняется.
Если $\alpha_k < 1$, то
\begin{equation*}
    \begin{aligned}
       f(x_{k+1}) - f^* & \leq f(x_k) - f^* - \frac{1}{2L_k} (\nabla f(x_k)^\top \hat{d_k})^2
       \\& \leq f(x_k) - f^* - \frac{(\nabla f(x_k)^\top d_k)^2}{2L_k D^2}
       \\&
       \leq f(x_k) - f^* - \frac{(f(x_k) - f^*)^2}{2L_k D^2}
       \\& = (f(x_k) - f^*) \left(1 - \frac{f(x_k) - f^*}{2 L_k D^2}\right).
    \end{aligned}
\end{equation*}
Если $\frac{f(x_k) - f^*}{2L_k D^2} \leq \frac{1}{k+3}$, то $f(x_k)-f^* \leq \frac{2L_k D^2}{k+3}$. Отсюда, с учётом $1 - \frac{(f(x_k) - f^*)}{2L_k D^2} \leq 1$, получим оценку
$$
    f(x_{k+1}) - f^* \leq \frac{2L_k D^2}{k+3} \leq \frac{2 D^2\max_{j \in \overline{0,k}} L_j }{k+3}.
$$
Если $\frac{f(x_k) - f^*}{2L_k D^2} > \frac{1}{k+3}$, то $f(x_k)-f^* > \frac{2L_k D^2}{k+3}$, откуда
\begin{equation*}
    \begin{aligned}
        f(x_{k+1}) - f^* &\leq \left(f(x_k)-f^*\right) \left(1-\frac{1}{k+3}\right)
        \\&
        \leq \left(\frac{2D^2 \max_{j \in \overline{0,k-1}} L_j }{k+2}\right)\left(\frac{k+2}{k+3}\right)
        \\& = \frac{2D^2 \max_{j \in \overline{0,k-1}} L_j}{k+3}
        \\& \leq \frac{2D^2 \max_{j \in \overline{0,k}} L_j}{k+3}.
    \end{aligned}
\end{equation*}

Неравенство \eqref{eq17} непосредственно вытекает из \eqref{eq12}.
\end{proof}

\begin{remark}\label{Rm0}
Если при некотором фиксированном $\Delta > 0$ верно $f(x_j)-f^* \geq \Delta \; \forall j \in \overline{0, k-1}$, то
\begin{equation}\label{eq18}
    f(x_k) - f^* \leq (f(x_0)-f^*) \prod\limits_{j=0}^{k-1} \phi_j,
\end{equation}
где
\begin{equation*}
    \phi_i =
    \begin{cases}
       \frac{1}{2}, & \text{ если } \alpha_i = 1,
       \\
       1-\frac{\Delta}{2L_iD^2}, & \text{ если } \alpha_i < 1.
    \end{cases}
\end{equation*}
\end{remark}

\begin{proof}
Неравенство \eqref{eq18} получается также индукцией по $k$. Для $k=1$ неравенство тривиально выполняется. Таким образом, базис индукции верен.

При $\alpha_k = 1$ шаг индукции очевидным образом следует из \eqref{eq17}. Если же $\alpha_k < 1$, то, применяя индуктивное допущение (соответствующую оценку для $ f(x_k) - f^*$), получим
\begin{equation*}
    \begin{aligned}
        f(x_{k+1}) & - f^*  \leq    \left(f(x_k)-f^*\right) \left (1-\frac{f(x_k)-f^*}{2L_k\|d_k\|^2} \right)
        \\&
        \leq \left(f(x_k)-f^*\right) \left (1-\frac{\Delta}{2L_k\|d_k\|^2} \right)
        \\&
        \leq (f(x_1)-f^*) \prod\limits_{j=0}^{k-1} \phi_j \left (1-\frac{\Delta}{2L_k\|d_k\|^2} \right)
        \\&
        \leq (f(x_1)-f^*) \prod\limits_{j=0}^{k} \phi_j.
    \end{aligned}
\end{equation*}
Таким образом, шаг индукции выполняется, что завершает доказательство \eqref{eq18}.
\end{proof}

Отметим ещё пару замечаний, описывающих потенциально возможные ситуации гарантированного уменьшения невязки по функции не менее чем в 2 раза.
\begin{remark}\label{Rm1}
Из \eqref{eq16} и \eqref{eq17} следует тот факт, что если последние $m$ итераций ($m < n$) верно, что $\alpha_k = 1$, то
$$
    f(x_{n}) - f^* \leq \frac{1}{2^{m-1}}\frac{D^2 \max_{j \in \overline{0,n-m-1}} L_j }{n-m+2}.
$$
\end{remark}

\begin{remark}\label{Rm2}
Также примечательно отметить факт, что если на $t\leq n$ итерациях алгоритма \ref{adaptive_FW} верно $\alpha_k = 1$, то
$$
    f(x_{n})-f^* \leq  \frac{f(x_0)-f^*}{2^t}.
$$
\end{remark}
Это следует из \eqref{eq17} и того, что при $\alpha_k < 1$
\begin{equation*}
    \begin{aligned}
        f(x_{k+1}) - f^* &\leq \left(f(x_{k}) - f^*\right)\left(1 - \frac{f(x_{k}) - f^*}{2L_{k} D^2}\right)
        \\& \leq f(x_{k}) - f^*.
    \end{aligned}
\end{equation*}
Заметим, что полученные оценки в \eqref{eq17} и примечаниях \ref{Rm1}, \ref{Rm2} являются отличительными особенностями предлагаемого в нашей работе подхода по сравнению с известными вариантами методов Франк--Вульфа \cite{Braun2022Conditional} при других подходах к выбору шагов.

Если предыдущие результаты для выпуклых задач дают только гарантии сублинейного скорости сходимости для выпуклых задач, то следующий подраздел посвящён более узкому классу задач, для которых можно установить сходимость предлагаемой адаптивной модификации метода Франк--Вульфа со скоростью геометрической прогрессии.

\subsection{Слyчай выпуклой целевой функций, удовлетворяющей условию градиентного доминирования}

Допустим, что помимо выпуклости целевая функция удовлетворяет условию градиентного доминирования Поляка-Лоясиевича. Оно введено Б.\,Т.\,Поляком в 1963 году в качестве условия (не требующего сильной выпуклости) но достаточного для доказательства глобальной линейной скорости сходимости градиентного спуска \cite{Polyak_condition} и является частным случаем неравенства С.\,Лоясиевича \cite{Lojasiewichz_inequality}. Напомним соответствующее понятие.

\begin{definition}
Дифференцируемая функция $f$ называется $c$-градиентно доминируемой для некоторой константы $c > 0$, если для всех $x$ и $y$ из ее области определения верно следующее неравенство:
$$
\frac{f(x) - f(y)}{c^2} \leq \|\nabla f(x)\|_*^2.
$$
\end{definition}

Среди примеров таких функций можно отметить сильно выпуклые функции, а также задачу логистической регрессии на любом компакте \cite{PL_Karimi}. Для дальнейшего изложения нам понадобится ещё одно определение.

\begin{definition}
Зазором оптимальности функции $ f: Q \longrightarrow \mathbb{R}$ в точке $x$ называется
$$
h(x) \overset{\mathrm{def}}{=} \max_{y \in Q} f(x) - f(y) = f(x) - f^*.
$$
\end{definition}

Далее, приведем утверждение, позволяющее доказать основную теорему параграфа.

\begin{proposition}\label{prop1}
Пусть $Q \subset \mathbb{R}^n$~--- компактное выпуклое множество диаметра $D$ и $f$~--- дифференцируемая выпуклая функция. Если существует радиус $r>0$, такой, что $B(x^*, r) \subset Q$, где $x^*$~--- точка минимума $f$, то для всех $x \in Q$ верно
$$
\langle \nabla f(x), x - v \rangle \geq r\| \nabla f(x) \|_* \geq \frac{r \langle \nabla f(x), x - x^* \rangle}{\|x-x^*\|},
$$
где $v = \arg\max\limits_{u \in Q} \langle \nabla f(x), x-u \rangle.$
\end{proposition}

Доказательство этого утверждения можно найти в \cite{Braun2022Conditional}.

\begin{theorem}
Пусть $Q$~--- компактное выпуклое множество диаметра $D$, а $f$~--- выпуклая $c$-градиентно доминируемая функция. Пусть также существует точка минимума $f$ $x^* \in Int(Q)$ во внутренности $Q$, т.~е. существует радиус $r > 0$, такой, что $B(x^*, r) \subset Q$. Тогда  для алгоритма \ref{adaptive_FW} при любых $k \geq 1$ верно
\begin{equation}\label{estimate_alg2}
f(x_k) - f^* \leq  (f(x_0)-f^*) \prod\limits_{i=1}^k \phi_i,
\end{equation}
где
\begin{equation*}
    \phi_i =
    \begin{cases}
       \frac{1}{2}, & \text{ если } \alpha_i = 1,
       \\
       1 - \frac{r^2}{L_i c^2 D^2}, & \text{ если } \alpha_i < 1.
    \end{cases}
\end{equation*}
\end{theorem}
\begin{proof}
Рассмотрим случай с $\alpha_k < 1$. Применяя предложение~\ref{prop1} с $x = x_k$ и $v = v_k$  и учитывая, что итерация алгоритма \ref{adaptive_FW} заканчивается тем, что выполняется неравенство из п.~7, для всех $k \geq 0$ получим $h_k - h_{k+1} = f(x_k) - f(x_{k+1}) \geq$
$$\frac{ (\nabla f(x_k)^\top (x_k-v_k))^2}{2L_k \|d_k\|^2} \geq $$
$$\geq  \frac{r^2 \|\nabla f(x_k)\|_*^2}{2L_k \|d_k\|^2} \geq \frac{r^2}{2L_k c^2 D^2} h_k.$$
Отсюда перегруппировкой слагаемых получим:
\begin{equation}\label{eq19}
    h_{k+1} \leq h_k \left(1 - \frac{r^2}{2L_k c^2 D^2}\right).
\end{equation}
Утверждение теоремы теперь получается комбинацией результатов \eqref{eq17} и \eqref{eq19} и очевидной индукции по $k$.
\end{proof}

\section{Численные эксперименты}

Для демонстрации производительности предложенного алгоритма \ref {adaptive_FW} были проведены численные эксперименты для ряда задач, как гладких, так и негладких. Как правило, метод условного градиента может не сходиться для негладких задач. Однако благодаря адаптивному подбору параметров $L_k$ в некоторых случаях алгоритм \ref{adaptive_FW} работает, и тогда становятся применимы результаты приведенных в данной работе теоретических результатов. Все вычисления были реализованы на Python 3.10.12 с помощью сервиса Google Colaboratory. Визуализация выполнялась с помощью библиотеки для визуализации данных Matplotlib. Для всех графиков масштаб осей был сделан логарифмическим.

1) Первый из рассматриваемых примеров естественно связан с хорошо известной задачей Ферма--Торричелли--Штейнера. Для этой задачи целевая функция имеет следующий вид:
\begin{equation}\label{obj_FTS}
    f(x):=\sum_{k=1}^N\left\|x-A_k\right\|_2.
\end{equation}

В качестве допустимого множества $Q$ для этой задачи используются следующие множества:
\begin{itemize}
    \item  $\ell_{\infty}$-шар с центром в $0 \in \mathbb{R}^n$ и радиусом $r > 0$, 
    то есть
    $$B_\infty(r) := \{x \in \mathbb{R}^n: \|x\|_{\infty} = \max_{1 \leq i \leq n}|x_i| \leq r  \};$$
    \item $\ell_p$-шар ($p \in [1, \infty)$) с центром в $0 \in \mathbb{R}^n$ и радиусом  $r>0$, то есть
    $$B_p(r) := \left\{ x \in \mathbb{R}^n: \|x\|_p^p = \sum_{i=1}^{n}|x_i|^p \leq r^p \right\}.$$
\end{itemize}

Точки $A_k,\;k = 1, \ldots, N$ выбирались случайным образом с нормальным (гауссовским) распределением с центром $0$ и стандартным отклонением, равным $1$. На $\ell_1$-шаре для $n = 1000$ и $r = 500$ адаптивный алгоритм \ref{adaptive_FW} показал лучшую сходимость и сошелся к значению $f = 297.47$, в то время как алгоритм с убывающим шагом сошелся к значению $f = 305.93$. Для $n = 100$ и $r = 10$ оба алгоритма показали примерно одинаковую эффективность. На $\ell_2$-шаре адаптивный алгоритм \ref{adaptive_FW} достигает хорошего качетсва приближённое решение всего за несколько итераций, в то время как алгоритм с убывающим шагом показал гораздо более медленную сходимость. На $\ell_{\infty}$-шаре алгоритмы показали примерно одинаковый результат. Однако можно отметить, что на первых итерациях адаптивный алгоритм сработал лучше (Рис.~\ref{Fig1}).

\begin{figure}[ht]
    \centering
    \includegraphics[width=4cm]{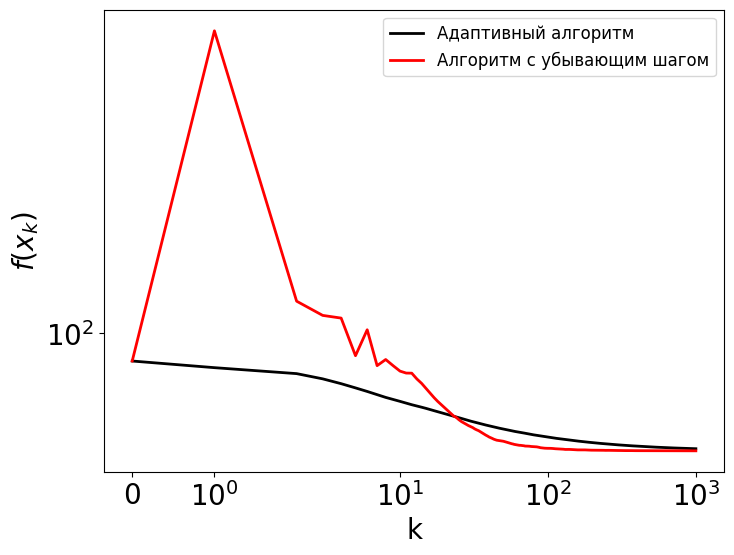}
    \includegraphics[width=4cm]{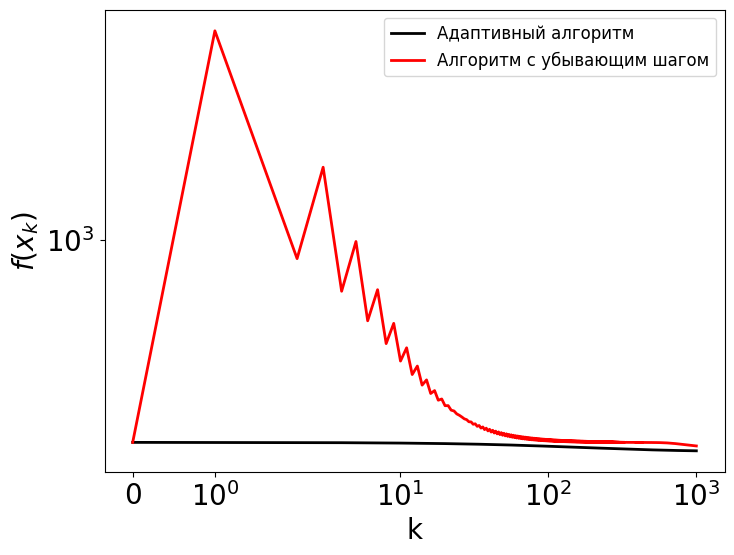}
    \includegraphics[width=4cm]{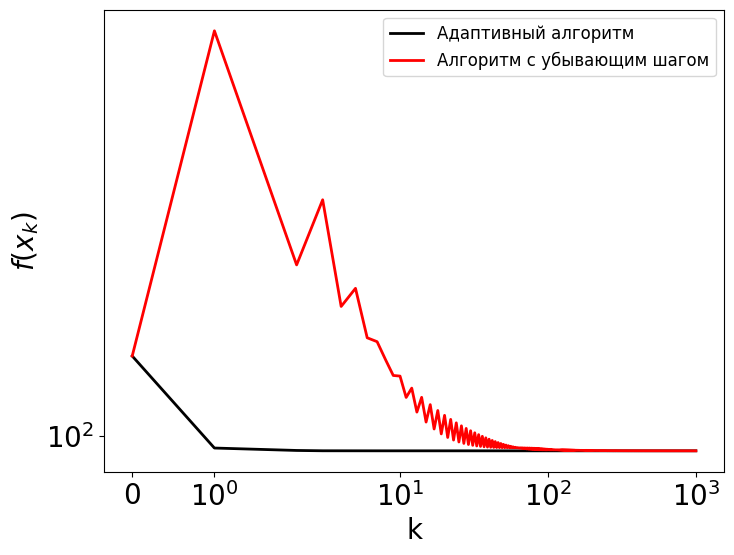}
    \includegraphics[width=4cm]{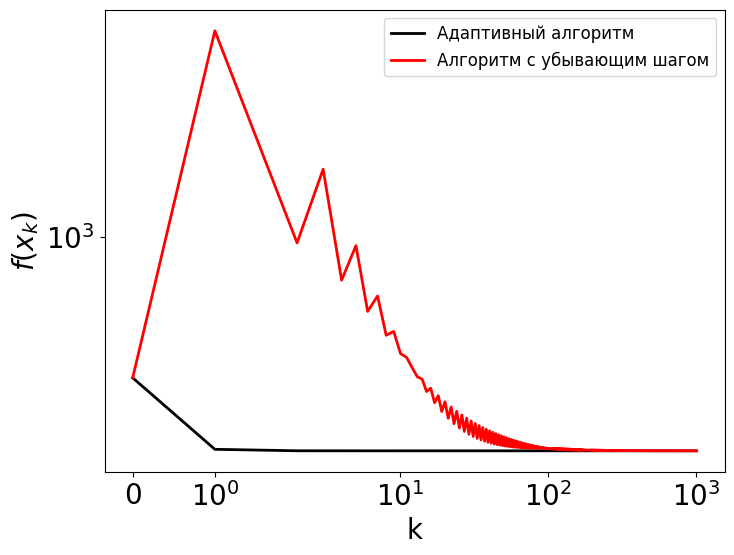}
    \includegraphics[width=4cm]{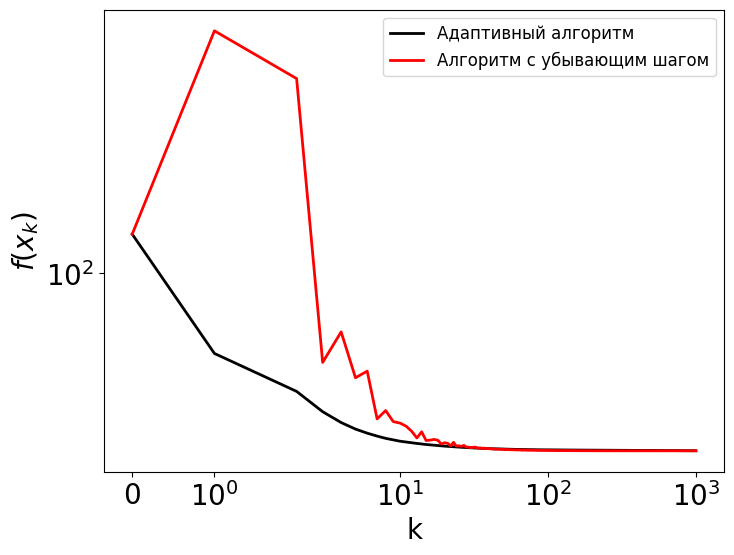}
    \includegraphics[width=4cm]{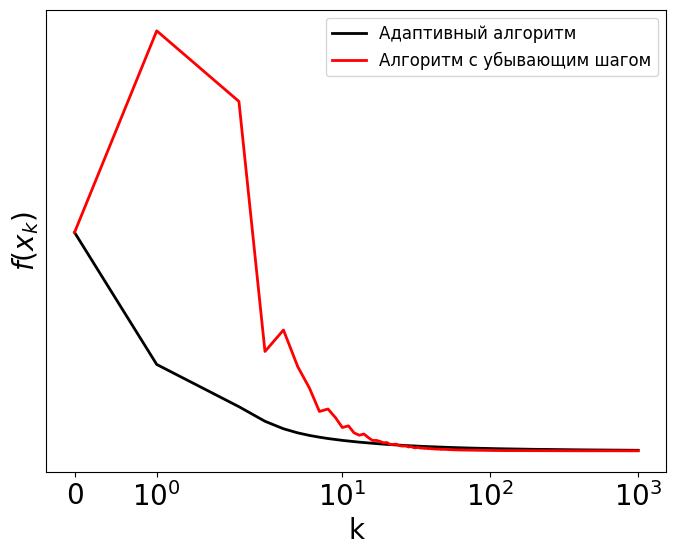}
    \caption{Результаты эксперимента для целевой функции \eqref{obj_FTS} на (сверху вниз) $\ell_1,\;\ell_2$ и $\ell_{\infty}$-шаре с $n=100$, 
    {$r=10$}
    (слева) и $n=1000, r=500$ (справа).}
    \label{Fig1}
\end{figure}

2) Вторая задача естественно связана с хорошо известной задачей о наименьшем покрывающем шаре. Для этой задачи целевая функция имеет следующий вид:
\begin{equation}\label{smallest_cov_quad}
	f(x) = \max_{1 \leq k \leq N} \|x - A_k\|_2^2.
\end{equation}
Точки $A_k \in \mathbb{R}^n, \; k = 1, \ldots, N$, выбирались случайным образом с нормальным (гауссовским) распределением с центром $0$ и стандартным отклонением, равным $1$.
\begin{figure}[ht]
    \includegraphics[width=4cm]{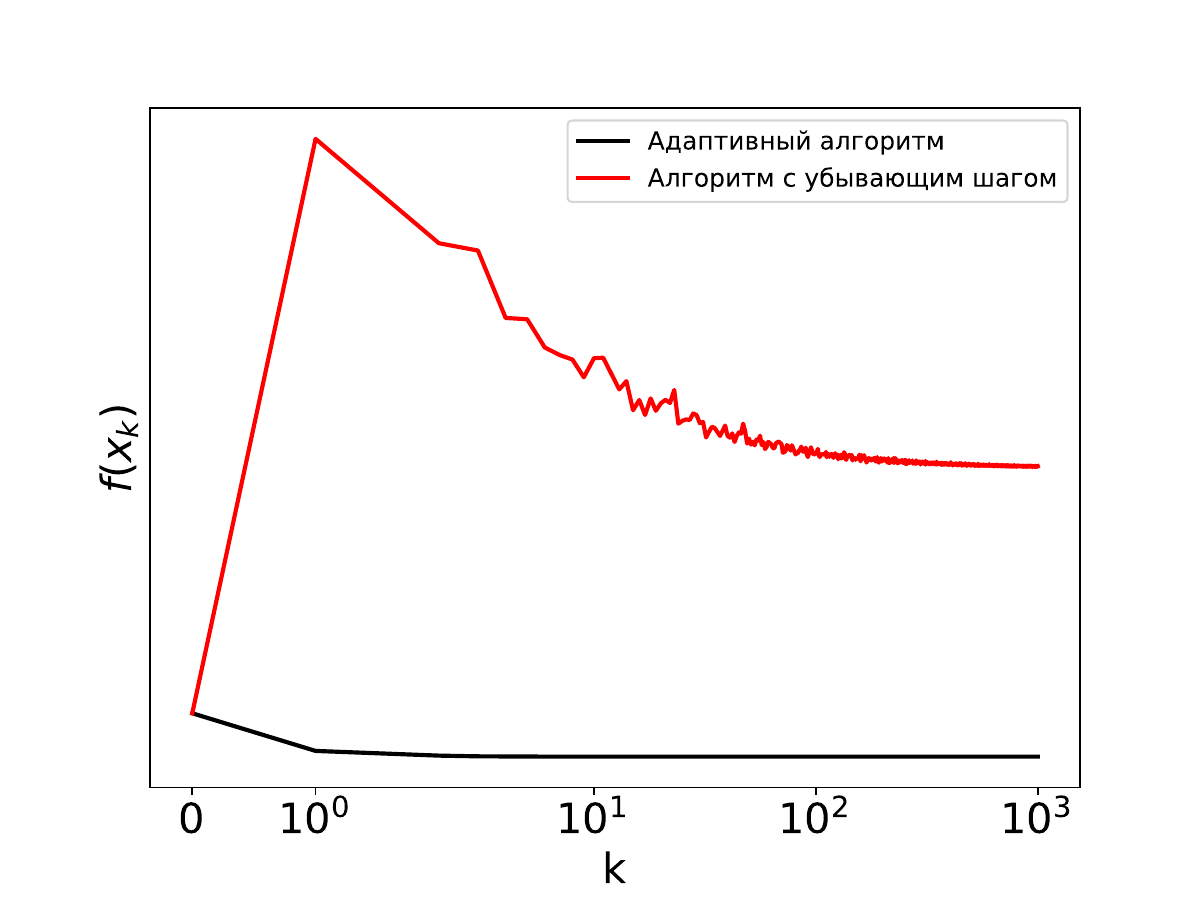}
    \includegraphics[width=4cm]{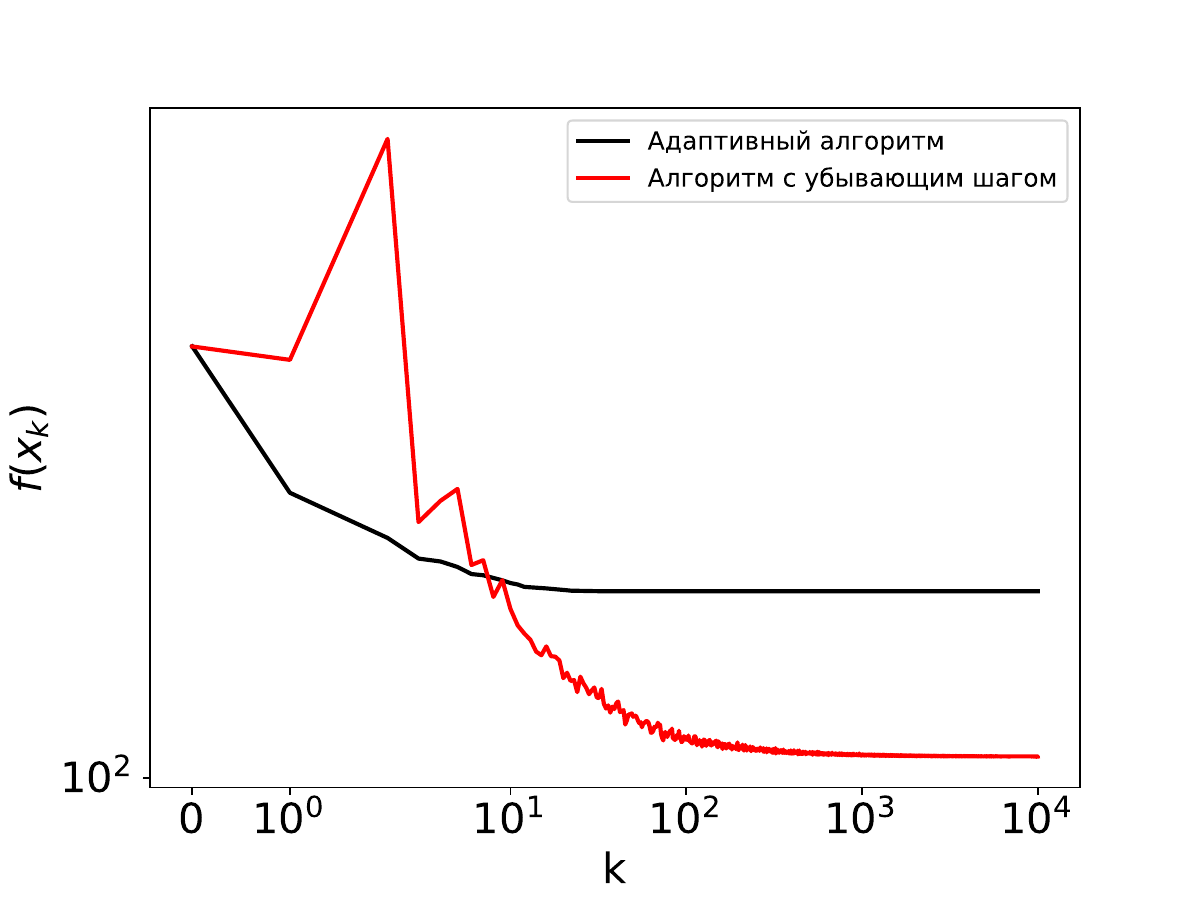}
    \centering
    \includegraphics[width=4.5cm]{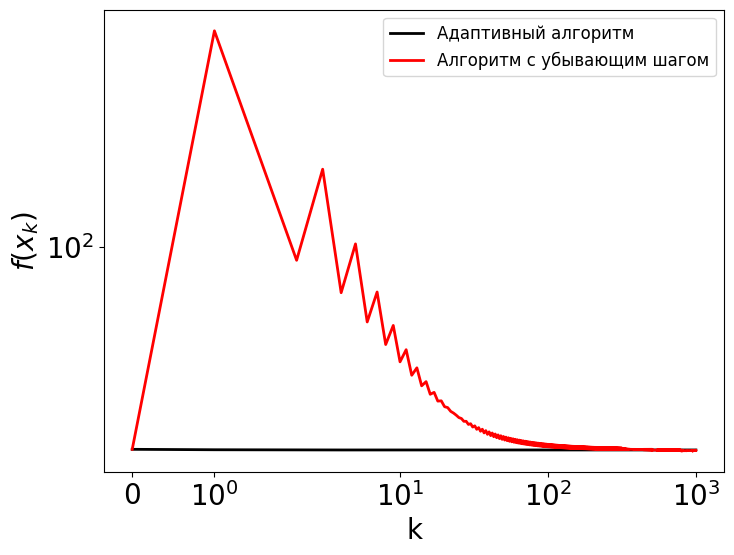}
    \caption{Результаты эксперимента для целевой функции \eqref{smallest_cov_quad} на  $\ell_2,\;\ell_{\infty}$-шаре (наверху слева направо) и $\ell_1$-шаре (внизу) с $n=1000, r=500$.}
    \label{Fig_smallest_covering}
\end{figure}
Для функции \eqref{smallest_cov_quad} на всех рассмотренных вариантах допустимых множеств адаптивный алгоритм на начальных итерациях сходится быстрее алгоритма с убывающим шагом, однако затем останавливается, так как для выхода из итерации константа $L_k$ становится настолько большой, что шаг $\alpha_k$ становится равным машинному нулю (Рис.~\ref{Fig_smallest_covering}).

3) Теперь для следующей гладкой функции
\begin{equation}\label{quad_function}
	f(x) = a_1 x_1^2 + \ldots a_n x_n^2,
\end{equation}
где $a_i >0 \; \forall i=1,\ldots, n$, мы увидим оценку \eqref{estimate_alg2}. Коэффициенты $a_1, \ldots, a_n$ выбраны равновероятно среди чисел $1, 2, \ldots, 10$. Результаты представлены на рисунке \ref{Fig_estimate}.
\begin{figure}[ht]
    \centering
    \includegraphics[width=7.0cm]{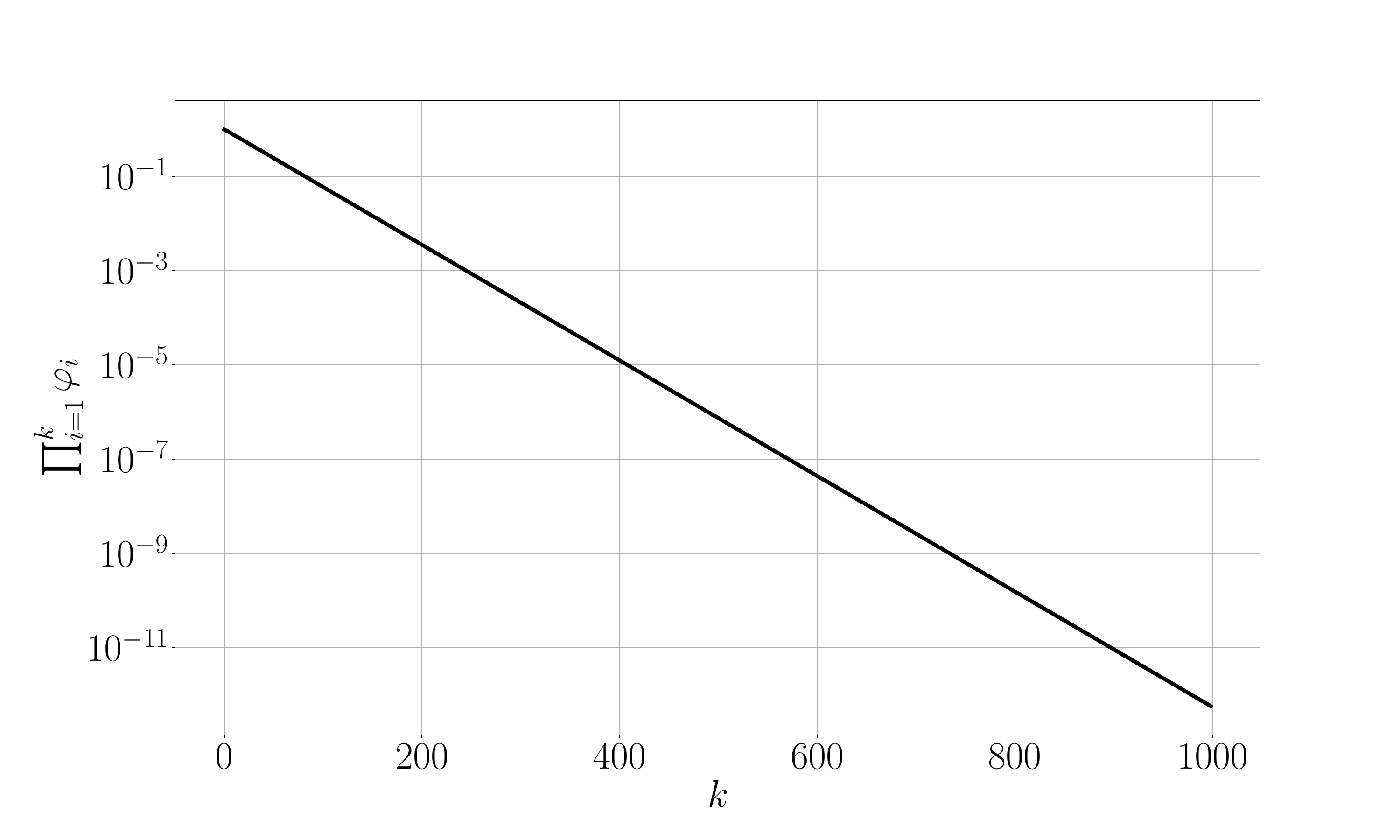}
    \caption{Оценка \eqref{estimate_alg2} для целевой функции \eqref{quad_function} на  $\ell_2$-шаре (единичный шар) с $n=1000$.}
    \label{Fig_estimate}
\end{figure}

4) Следующая задача заключается в восстановлении малоранговой матрицы из разреженного набора наблюдаемых матричных входов $\{U_{ij}\}_{(i,j)\in J}$, где $J \subseteq \{1, \ldots, n_1\} \times \{1, \ldots, n_2\}$. Таким образом, задача может быть сформулирована следующим образом:
\begin{equation}\label{matrix_completion}
\min_{X \in \mathbb{R}^{n_1 \times n_2}} f(X) := \sum\limits_{(i, j) \in J} (X_{ij} - U_{ij})^2,
\end{equation}
где
$$
\operatorname{rank}(X) \leq \delta.
$$
Здесь параметр $\delta > 0$ характеризует предположение о ранге восстанавливаемой матрицы. На практике для упрощения ограничение на ранг матрицы заменяется ограничением на ядерную норму матрицы, где ядерная норма $\|X\|_*$ матрицы $X$ равна сумме ее сингулярных значений. Таким образом, мы получаем следующую задачу выпуклой оптимизации:
\begin{equation}\label{matrix_completion1}
\min_{X \in \mathbb{R}^{n_1 \times n_2}} f(X) := \sum\limits_{(i, j) \in J} (X_{ij} - U_{ij})^2,
\end{equation}
где $\|X\|_* \leq \delta.$

Допустимое множество представляет собой выпуклую оболочку матриц единичного ранга
\begin{equation*}
\begin{aligned}
&Q = \{X \in \mathbb{R}^{n_1 \times n_2}: \|X\|_* \leq \delta \}
\\&
=\operatorname{conv}\{\delta u v^\top: u \in  \mathbb{R}^{n_1},\;v \in \mathbb{R}^{n_2},\;\|u\| = \|v\| = 1\}.
\end{aligned}
\end{equation*}
Подробнее про задачу см. \cite{matrix_completion_ref}.

Эксперименты были проведены на сгенерированных данных, а также на датасете MovieLens, который представляет из себя матрицу оценок пользователей фильмов \cite{movielens100K}.

Искусственные данные были сгенерированы в соответствии с моделью $X = \omega_1 U V^\top + \omega_2 \mathcal{E}$, где элементы матриц $U \in \mathbb{R}^{m \times r}$, $V \in \mathbb{R}^{n \times r}$ и $\mathcal{E} \in \mathbb{R}^{m \times n}$ выбраны случайным образом с нормальным распределением с матожиданием $0$ и стандартным отклонением, равным $1$, а скалярные параметры $\omega_1$, $\omega_2$ отвечают за соотношение сигнала и шума (SNR), а именно: $\omega_1 := \frac{1}{\|U V^\top\|_F}$, а $\omega_2 := \frac{1}{(SNR \|\mathcal{E}\|_F)}$, где $\|\cdot\|_F$ обозначает норму Фробениуса матрицы, а именно: для матрицы $A$ размера $m \times n$ $\|A\|_F = \sqrt{\sum\limits_{i=1}^m \sum\limits_{i=1}^n |a_{ij}|^2}$. Множество наблюдаемых элементов $\Omega$ было получено с использованием равномерного случайного сэмплирования элементов с вероятностью $p$, где $p$~--- это желаемая доля наблюдаемых элементов. Для эксперимента было выбрано $m=250$, $n=200$.
Используемый датасет MovieLens состоит из 647 строк, соответствующих пользователям и 9300 столбцов, соответствующих фильмам, однако для ускорения вычислений были выбраны только первые 1500 столбцов. На описанной задаче для сгенерированных данных при различных значениях параметров адаптивный алгоритм \ref{adaptive_FW} продемонстрировал лучшую скорость сходимости, чем неадаптивный вариант с постоянным $L$, и показал примерно одинаковую эффективность с алгоритмом с убывающим шагом (Рис.~\ref{fig4}).
\begin{figure}[ht]
    \centering
    \includegraphics[width=4cm]{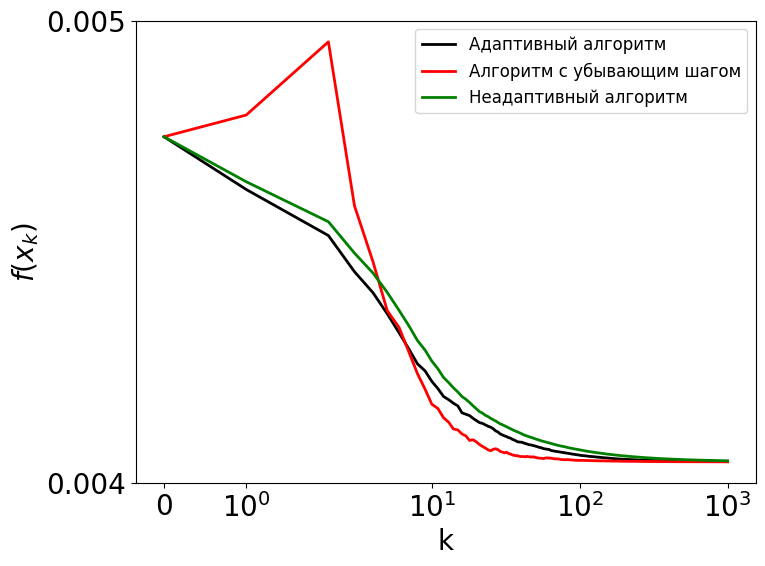}
     \includegraphics[width=4cm]{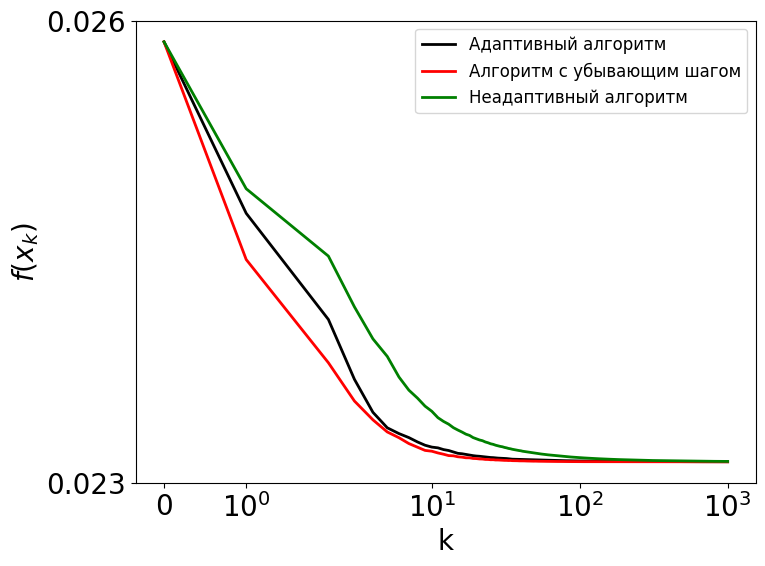}
    \includegraphics[width=4cm]{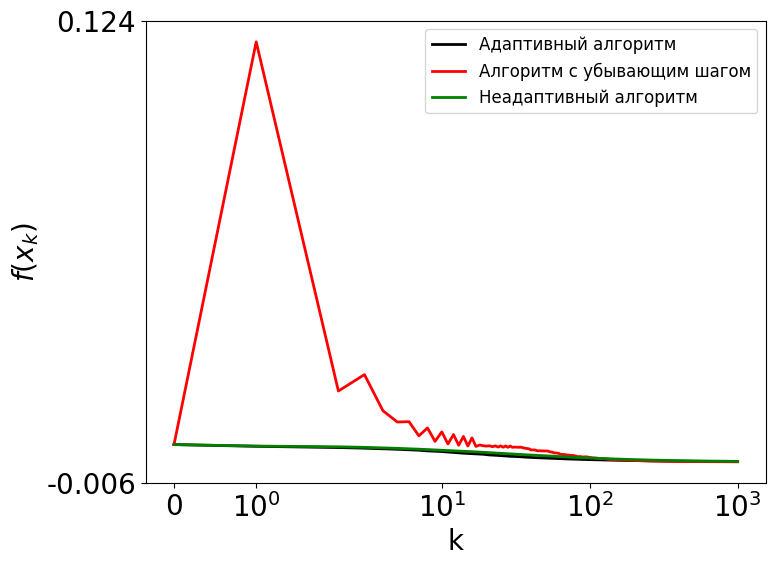}
    \includegraphics[width=4cm]{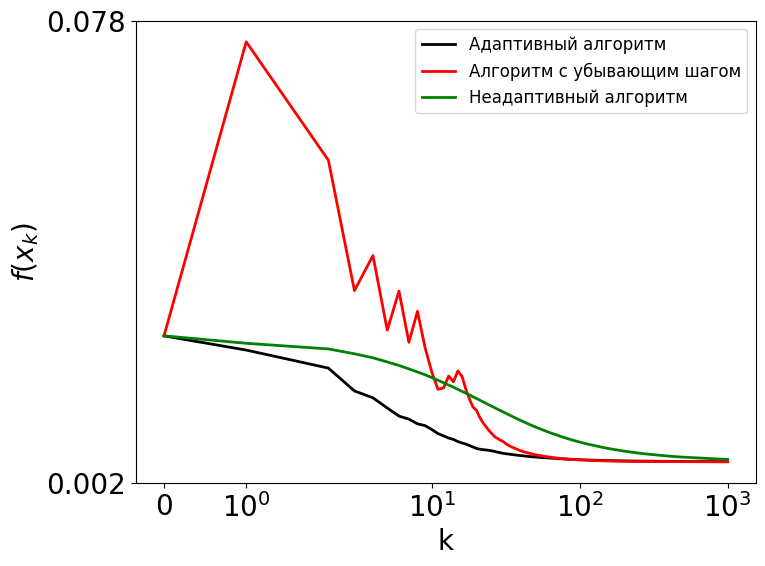}
    \caption{Результаты эксперимента для целевой функции \eqref{matrix_completion} с (слева направо) $d=0.08$, $r=10$, $p=0.01, SNR = 4$; $d=0.08$, $r=10$, $p=0.05$, $SNR = 4$; $d=0.8$, $r=3$, $p=0.01$, $SNR = 4$; $d=0.8$, $r=3$, $p=0.05$, $SNR = 4$.}
    \label{fig4}
\end{figure}
А на датасете с фильмами адаптивный алгоритм \ref{adaptive_FW} продемонстрировал лучшую динамику, чем неадаптивный алгоритм с убывающим шагом, и сошелся к значению $f \approx 28601.28$, в то время как алгоритм с убывающим шагом сошелся к значению $f \approx 28635.68$, а неадаптивный~--- к $f \approx 29604.09$ (Рис.~\ref{fig_5}).
\begin{figure}[ht]
    \centering
    \includegraphics[width=6cm]{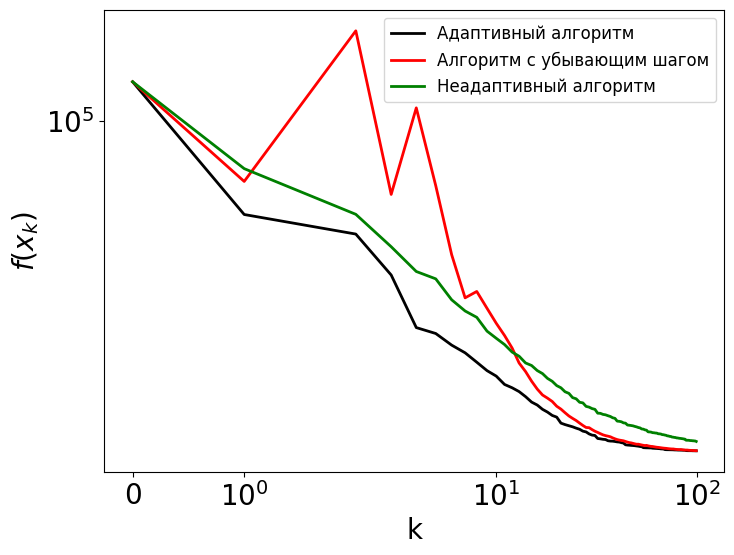}
    \caption{Результаты эксперимента для целевой функции \eqref{matrix_completion}.}
    \label{fig_5}
\end{figure}

5) Следующий пример связан с методом опорных векторов. Методы опорных векторов представляют собой важный класс инструментов машинного обучения (подробнее, например, тут \cite{SVM}). Дан размеченный набор точек, обычно называемый тренировочным набором,
\begin{equation}\label{TS}
    TS = \{(p_i, y_i), p_i \in \mathbb{R}^d, y_i \in \{-1, 1\}, i = 1, \ldots,n \}.
\end{equation}
Метод опорных векторов состоит в нахождении линейного классификатора $w \in \mathbb{R}^d$, такого, что метка $y_i$ может быть выведена с <<наибольшей уверенностью>> из $w^\top p_i$. Выпуклая квадратичная формулировка этой задачи следующая \cite{Clarkson_SVM}:
\begin{equation}\label{SVM_quadratic}
\min_{w \in \mathbb{R}^d, \text{ } \rho \in \mathbb{R}} \rho + \frac{||w||^2}{2},
\end{equation}
где
$$
\rho + y_i w^\top p_i \geq 0 \quad \forall i = 1, \ldots, n,
$$
и $\rho < 0$ может быть тогда и только тогда, когда существует точный линейный классификатор, то есть такой, что $w^\top p_i = \operatorname{sign}(y_i)$. Двойственная к \eqref{SVM_quadratic} снова является стандартной задачей квадратичного программирования:
\begin{equation}\label{SVM}
    \min_{x \in \triangle_{n}} x^\top A^\top A x,
\end{equation}
где $A = (y_1 p_1\; \dots\; y_n p_n)$, а $\triangle_{n}$~--- единичный $n$-мерный симплекс, определяемый как
$$
    \triangle_{n} := \left\{ x \in \mathbb{R}^n: 0 \leq x \leq 1 \text{ и } 1^\top x = 1 \right\}.
$$
В качестве данных для эксперимента использовался датасет из Национального института диабета, заболеваний органов пищеварения и почек, который представляет из себя набор некоторых медицинских измерений для каждого пациента и бинарную целевую переменную~--- информацию о наличии у пациента диабета \cite{diabet_dataset}. На этой задаче адаптивный алгоритм \ref{adaptive_FW} показал сyщественно лучшую динамику, чем два других алгоритма, и примерно за 500 итераций сошелся к значению $f \approx 0.01$, в то время как алгоритм с убывающим шагом за 100000 итераций сошелся лишь к значению $f \approx 0.146$, а неадаптивный алгоритм~--- к $f \approx 0.155$ (Рис.~\ref{fig_6}).
\begin{figure}[ht]
    \centering
    \includegraphics[width=5.5cm]{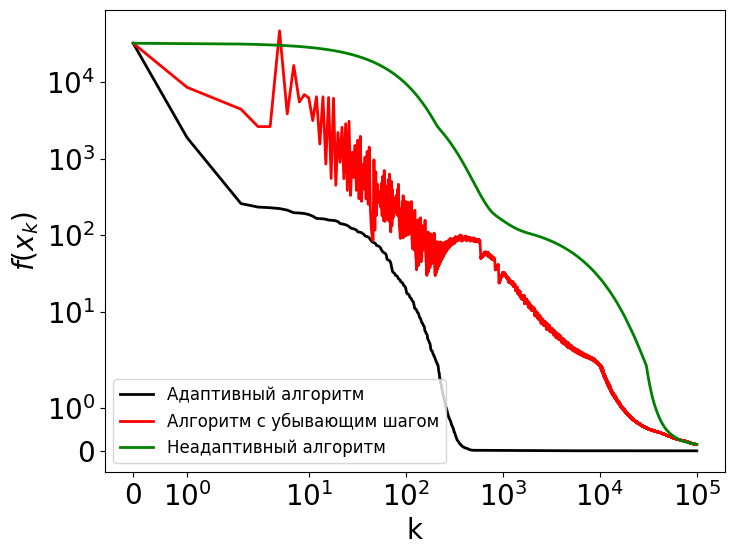}
    \caption{Результаты эксперимента для целевой функции \eqref{SVM}.}
    \label{fig_6}
\end{figure}

6) Наконец, рассмотрим задачу обучения логистической регрессии для задачи классификации. Ей соответствует целевая функция эмпирического риска с бинарной кросс-энтропией в качестве функции потерь:
\begin{align}\label{logreg}
\nonumber f(w) := -\frac{1}{m} \sum_{i=1}^m (y_i \ln \hat{y}_i(w) + \\
\nonumber  + (1 - y_i) \ln (1 - \hat{y}_i(w))),\\
\hat{y}_i(w) := \frac{1}{1 + \exp (-\langle X_i, w\rangle)},\quad i = 1, \ldots, m,
\end{align}
в которой $\hat{y}_i \in (0, 1]$ имеет смысл предсказания моделью логистической регрессии, параметризованной весами $w \in \mathbb{R}^n$, целевой переменной, $X_i \in \mathbb{R}^n$~--- вектор признакового описания $i$-го объекта набора данных, $y_i \in \{0, 1\}$~--- значение целевой переменной для $i$-го объекта, $m$~--- число объектов в наборе данных. В качестве дополнительного ограничения, накладываемого на $w$, требуется выполнение условия $w \in B_2(r)$, $r > 0$, которое называется регуляризацией Иванова и эквивалентно $\ell_2$-регуляризации Тихонова \cite{ivanov}. В качестве набора данных для этого эксперимента использовался a1a из библиотеки LibSVM \cite{libsvm}.
\begin{figure}[ht!]
    \centering
    \includegraphics[width=4cm]{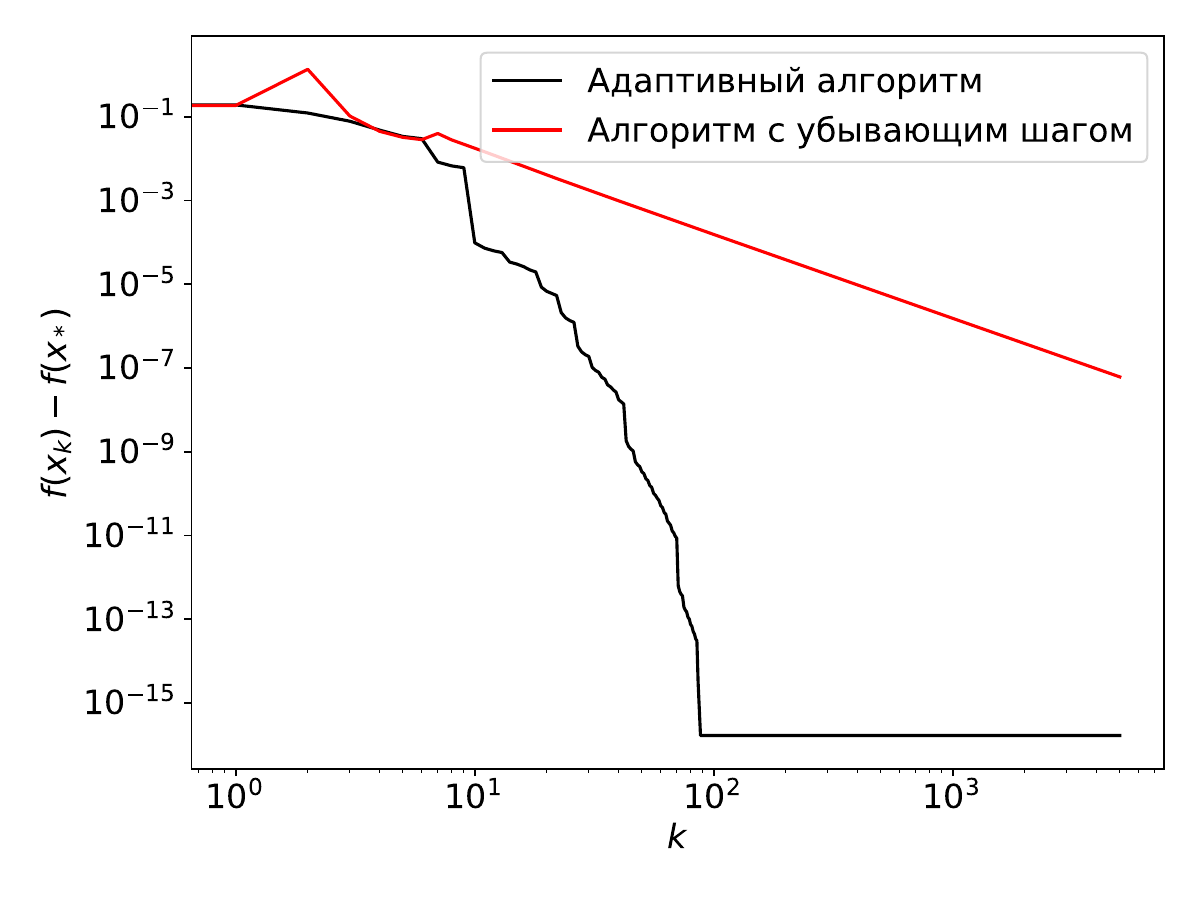}
    \includegraphics[width=4cm]{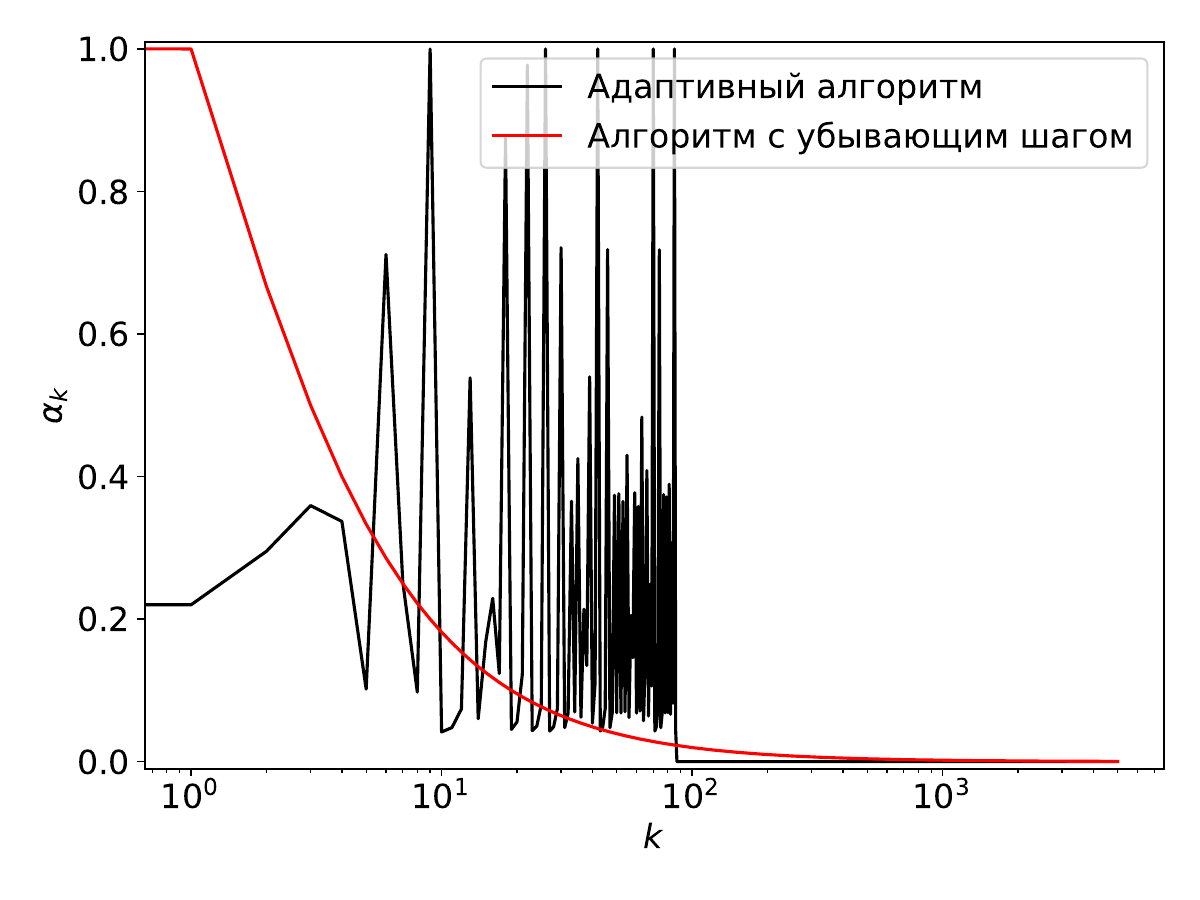}
    \includegraphics[width=4cm]{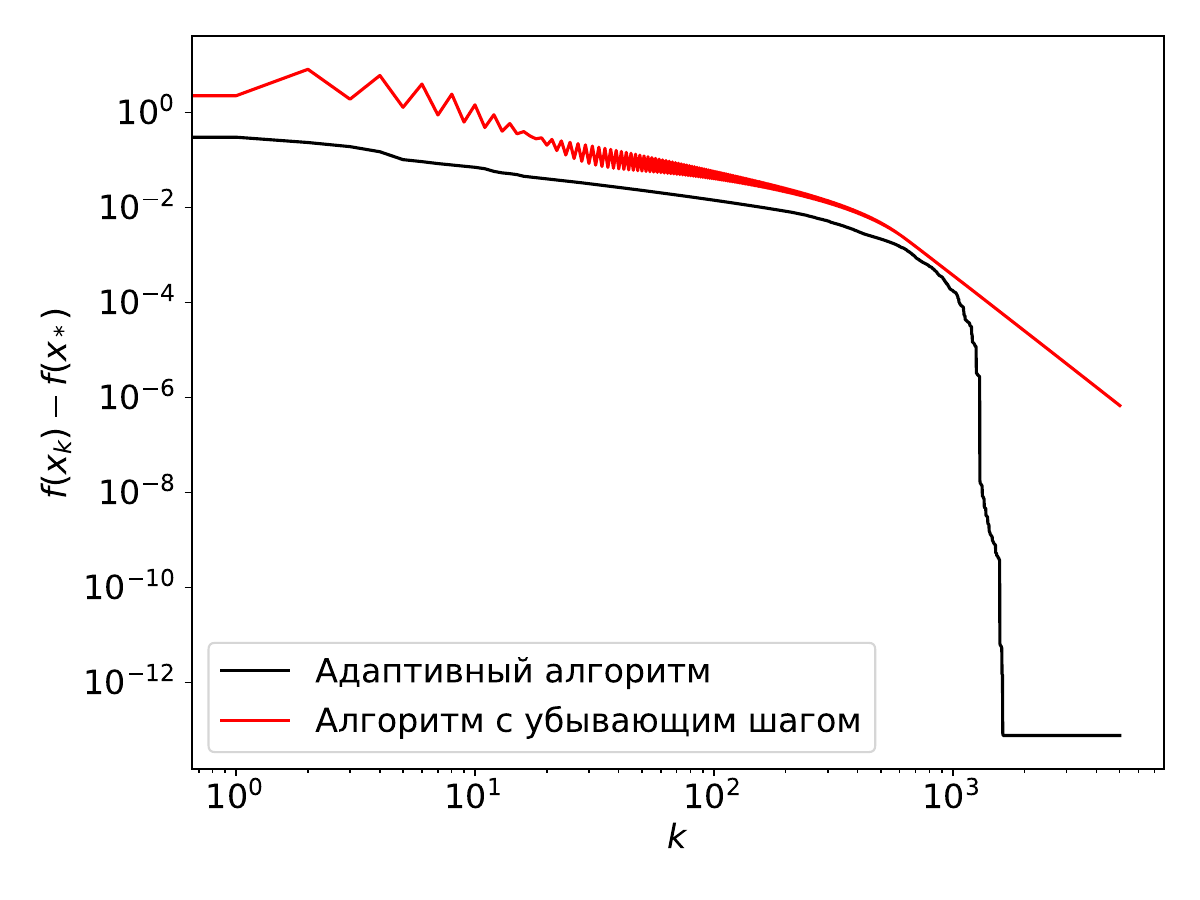}
    \includegraphics[width=4cm]{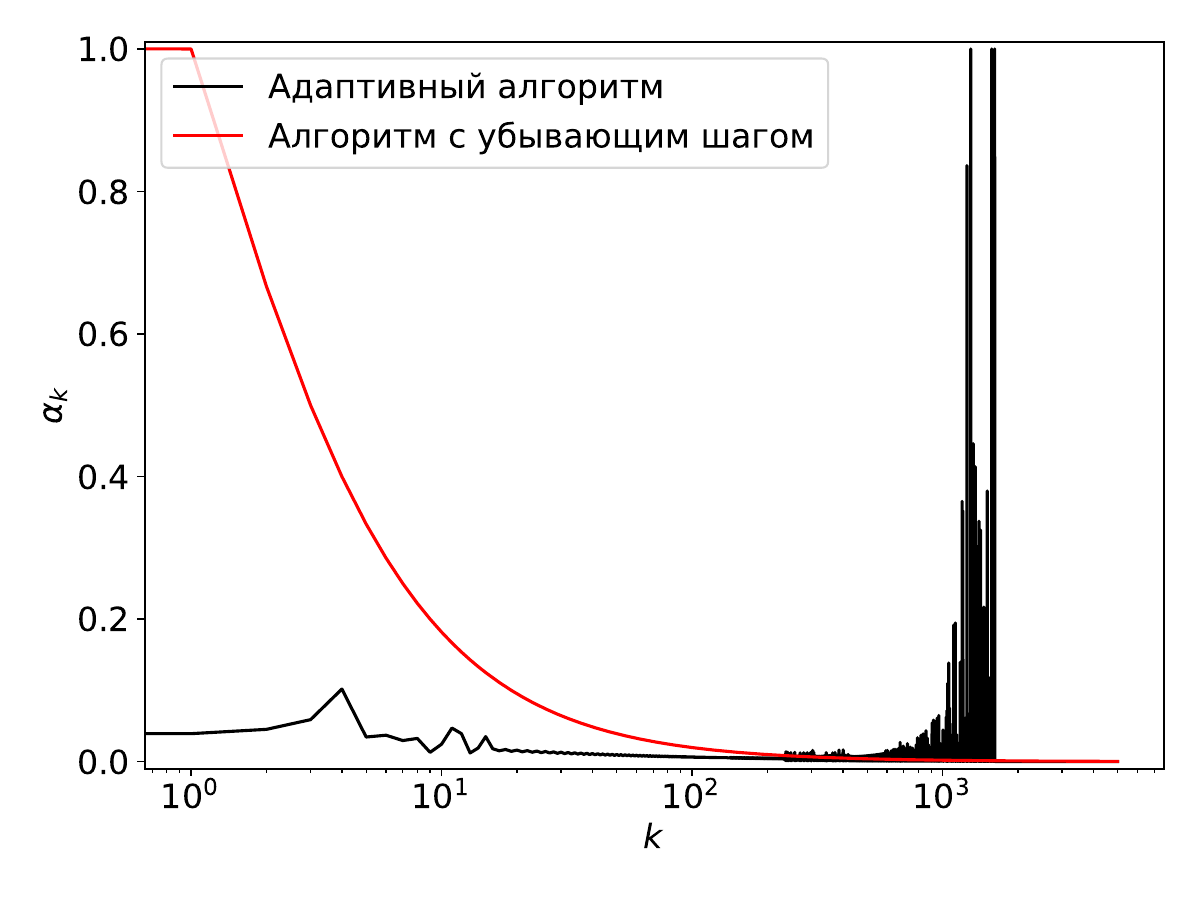}
    \caption{Результаты эксперимента для целевой функции \eqref{logreg} на $\ell_2$-шаре с $n=119$, $m=1605$, $r=1$ (сверху) и $r=5$ (снизу).}
    \label{fig:logreg}
\end{figure}

Для этой задачи адаптивный алгоритм \ref{adaptive_FW} позволяет достигнуть решения задачи с точностью до неустранимой машинной погрешности за существенно меньшее число итераций, чем неадаптивный алгоритм (левая часть рис.~\ref{fig:logreg}). Как можно видеть из графика адаптивно выбираемых значений параметра $\alpha_i$ (правая часть рис.~\ref{fig:logreg}), чем ближе текущая точка к решению задачи, тем чаще удаётся установить $\alpha_i$ равным $\frac{1}{2}$, вследствие чего гарантировать уменьшение невязки по функции вдвое при соответствующем шаге и ускорить приближение к решению. Свобода выбрать более эффективный коэффициент $\alpha_i$ на $i$-той итераций появляется благодаря адаптивности алгоритма~\ref{adaptive_FW}.

\section*{Заключение}
В статье исследован адаптивный вариант метода Франк--Вульфа для задач выпуклой минимизации. Его сходимость была обоснована с точки зрения адаптивно подбираемых параметров, соответствующих константе Липшица градиента. Детально проанализирован случай, когда на итерациях алгоритма возможно теоретически гарантировать уменьшение невязки не менее чем в два раза. Для выпуклых гладких задач возможно доказать гарантии сходимости метода с оптимальной cублинейной скоростью. Сходимость метода со скоростью геометрической прогрессии для метода Франк--Вульфа возможна, как известно, при дополнительных предположениях о целевой функции или допустимом множестве. Соответственно, в статье получены также оценки скорости сходимости предлагаемого алгоритма с адаптивно подбираемыми параметрами $L_k$ в ситуации, когда помимо выпуклости предполагается, что целевая функция удовлетворяет условию градиентного доминирования.

Проведённые эксперименты показали, что адаптивный алгоритм \ref{adaptive_FW} может приводить к лучшим результатам по сравнению со стандартным шагом, зависящим от константы Липшица градиента целевой функции для рассмотренных типов выпуклых гладких задач. Из результатов экспериментов также видно, что зачастую адаптивный алгоритм не уступает в эффективности классическому алгоритму с убывающим шагом, а часто работает намного эффективнее него, причём как для гладких, так и для негладких задач. Так полyчилось, например, для задачи Ферма--Торричелли--Штейнера, являющейся негладкой, для $\ell_1$-шара с $n=1000$, $r=500$ и $\ell_2$-шара с $n=1000$, $r=500$ и $n=100$, $r=10$. А в некоторых случаях, как, например, в задачах, связанных с методом опорных векторов или логистической регрессией, которые являются гладкими, адаптивный алгоритм сходится на порядок быстрее классического алгоритма.

\section*{Финансирование} Работа выполнена при поддержке гранта Российского научного фонда, проект № 21-71-30005, \url{https://rscf.ru/project/21-71-30005/}.

\end{document}